\documentclass[12pt]{article}

\newlength{\stefan}
\usepackage{amsmath, amsthm, amsfonts, makeidx, amscd,graphicx,epsfig}
\usepackage{multind}
\DeclareMathSymbol{\subsetneq}{\mathord}{AMSb}{"26}

\newtheorem{lemma}{Lemma}[section]

\newtheorem{theorem}[lemma]{Theorem}

\newtheorem{proposition}[lemma]{Proposition}

\theoremstyle{definition}

\newtheorem{example}[lemma]{Example}

\newcommand{\lp}{\longrightarrow}

\newcommand{\lb}{\longleftarrow}

\newcommand{\mb}{\mathbb}

\newcommand{\G}{\mathcal{G}}

\newcommand{\C}{\mb{C}}

\newcommand{\N}{\mb{N}}

\newcommand{\DER}{\textup{DER}}

\renewcommand{\ker}{\mathit{ker}}

\newcommand{\trdeg}{\mathit{trdeg}}
\newcommand{\LND}{\textup{LND}}

\newcommand{\quot}{/\!\!/}
\newcommand{\codim}{\operatorname{codim}}
\newcommand{\id}{\operatorname{id}}
\newcommand{\pr}{\operatorname{pr}}
\renewcommand{\O}{\mathcal{O}}

\title{A commuting derivations theorem on UFD's}
\author{H.Derksen\footnote{Funded by NSF CAREER grant,  DMS 0349019,
Invariant Theory, Algorithms and Applications.},\\
University of Michigan, MI, USA\\
hderksen@umich.edu\\
\ \\
 A.van den Essen\\
 Radboud University, Nijmegen, The Netherlands\\
 a.vandenessen@math.ru.nl\\
 \ \\
 S.Maubach\footnote{Funded by Veni-grant of council for the
physical sciences, Netherlands Organisation for scientific research (NWO)} (corresponding author)\\
Radboud University, Nijmegen, The Netherlands\\
s.maubach@math.ru.nl\\}

\begin{document}
\maketitle

\abstract{Let $A$ be the polynomial ring over $k$ (a field of characteristic zero) in $n+1$ variables.
The commuting derivations conjecture states that $n$ commuting locally nilpotent derivations on $A$,
linearly independent over $A$, must satisfy $A^{D_1,\ldots,D_m}=k[f]$ where $f$ is a coordinate.
The conjecture can be formulated as stating that a
$(G_m)^n$-action on $k^{n+1}$ must have invariant ring $k[f]$ where $f$ is a coordinate.
In this paper we prove a statement (theorem \ref{CDH2}) where we assume less on $A$ ($A$ is a {\sc UFD} over $k$ of transcendence degree $n+1$ satisfying $A^*=k$)
and prove less ($A/(f-\alpha)$ is a polynomial ring  for all but finitely many $\alpha$).
Under certain additional conditions (the $D_i$ are linearly independent modulo $(f-\alpha)$ for each $\alpha\in k$)
we prove that $A$ is a polynomial ring itself and $f$ is a coordinate. This statement is proven even more generally by replacing
``free unipotent action of dimension $n$'' for
``$G_a^n$-action''.

We make links with the (Abhyankar-)Sataye conjecture and give a new equivalent formulation of the Sataye conjecture. }

\section{Preliminaries and introduction}

{\bf Notations:}
$k$ will denote a field of characteristic zero.
For a $k$-algebra $A$ we define $\LND(A)$ as the set of all locally nilpotent derivations, and $\DER(A)$ as the set of derivations.
We will denote by $A^{D_1,\ldots,D_m}:=\{a\in A; D_1(a)=\ldots=D_m(a)=0\}$.\\

In the paper \cite{Mau}, the following conjecture is posed:\\

\noindent
{\bf Commuting Derivations Conjecture:} Let $A:=k[X_1,\ldots,X_{n+1}]$, and let $D_1,\ldots,D_n\in LND(A)$ be commuting, linearly independent over $A$, locally nilpotent derivations. Then $A^{D_1,\ldots,D_n}=k[f]$ and $f$ is a coordinate. \\
{\bf Geometric version:} Suppose we have a $\G:=(\G_a)^n$ -action on $k^{n+1}$. Then $k[X_1,\ldots,X_{n+1}]^{\G} =k[f]$ and $f$ is a coordinate.\\

In the elegant paper \cite{ElKa05}, it is shown that this conjecture is equivalent to the following:\\

\noindent
{\bf Weak Abhyankar-Sataye Conjecture:} Let $A:=k[X_1,\ldots,X_{n+1}]$, and let $f\in A$ be such that $k(f)[X_1,\ldots,X_n]\cong_{k(f)} k(f)[Y_1,\ldots,Y_{n-1}]$.
Then $f$ is a coordinate in $A$.\\

For completeness sake, let us state\\

\noindent
{\bf Abhyankar-Sataye Conjecture:} Let $A:=k[X_1,\ldots,X_{n+1}]$, and let $f\in A$ be such that $A/(f)\cong k[Y_1,\ldots,Y_n]$.
Then $f$ is a coordinate.\\

\noindent
{\bf Sataye Conjecture:} Let $A:=k[X_1,\ldots,X_{n+1}]$, and let $f\in A$ be such that $A/(f-\alpha)\cong k[Y_1,\ldots,Y_n]$ for all $\alpha\in\C$.
Then $f$ is a coordinate.\\

In \cite{Mau}, the Commuting Derivations Conjecture is proven for $n=3$.
But there is no indication that it might be true in higher dimensions. Even more, the V\'en\'ereau polynomials (see\cite{Ven}) (or similar objects), which are
candidate counterexamples to the Abhyankar-Sataye conjecture, could very well spoil things for the Commuting Derivations Conjecture in higher dimensions.
In any case, it seems like a proof is far away.

Therefore, it seems a good idea to be a little less ambitious.
in this paper, we consider the weaker statement that $A$ is a UFD (in stead of a polynomial ring). 
It turns out that the situation can be quite different and interesting. 
Let us consider a famous example:

\begin{example} \label{Ex1} Let $A:=\C[x,y,z,t]=\C[X,Y,Z,T]/(X^2Y+X+Z^2+T^3)$ and let $D_1:=2z\frac{\partial}{\partial y}-x^2\frac{\partial}{\partial z}$ and
$D_2:=3t^2\frac{\partial}{\partial y}-x^2\frac{\partial}{\partial t}$. $A$ is a {\sc ufd} of transcendence degree 3 which is not a polynomial ring 
(see \cite{ML}, or use the fact that the commuting derivations conjecture in dimension 3 holds).
$D_1$ and $D_2$ commute, and $A^{D_1,D_2}=\C[x]$. Now $A/(x-\alpha)\cong \C[Y_1,Y_2]$ except in the case that $\alpha=0$.
Also, $D_1 \mod(x-\alpha), D_2\mod(x-\alpha)$ are independent over $A/(x-\alpha)$ if and only if $\alpha\not = 0$.
\end{example}

\section{The {\sc ufd} Commuting derivations theorem}

The following theorem is the main result of this paper.

\begin{theorem}\label{CDH2a}\label{CDH2}
Let $A$ be a {\sc ufd} over $k$ with $\trdeg_kQ(A)= n+1 (\geq 1)$, $A^*=k^*$,   and let $D_1,\ldots,D_n$ be commuting
locally nilpotent derivations (linearly independent
over $A$). Now $A^{D_1,\ldots,D_n}=k[f]$ for some $f\in A\backslash k$, and
\begin{enumerate}
\item
If $D_1\mod(f-\alpha),\ldots, D_n\mod(f-\alpha)$ are independent over $A/(f-\alpha)$, then $A/(f-\alpha)\cong \C^{[n]}$.
There are only finitely many $\alpha\in\C$ for which $D_1\mod(f-\alpha),\ldots, D_n\mod(f-\alpha)$
are dependent over $A/(f-\alpha)$.
\item
In the case that $D_1\mod(f-\alpha),\ldots, D_n\mod(f-\alpha)$ are
independent over $A/(f-\alpha)$ for each $\alpha\in k$, then
$A=k[s_1,\ldots, s_n,f]$, a polynomial ring in $n+1$ variables.
\end{enumerate}
\end{theorem}
\noindent
{\bf Geometric Version:} {\em Let $V$ be a factorial affine surface over $k$ of
dimension $n+1$ such that $\mathcal{O}(V)^*=k^*$. Suppose there exists a
$\G:=(\G_a)^n$-action on $V$. Then $\mathcal{O}(V)^{\G}=k[f]$ and
\begin{enumerate}
\item
Suppose that the fiber $f=\alpha$ has a point with trivial stabilizer. Then the fiber $f=\alpha$ is isomorphic to $\C^n$.
There are only finitely many $\alpha$ for which $f=\alpha$ has no point with trivial stabilizer.
\item
Suppose that all fibers $f=\alpha$ have a point with trivial stabilizer. (Then, {\em all} points have trivial stabilizers.)
Then $V\cong \C^{n+1}$ and the action $\G\times V\lp V$ is a translation on the first $n$ coordinates.
\end{enumerate}}

In the last section we will prove a more general geometric statement of part 2 for unipotent groups in stead of $\G_a^n$-actions, but we will stick
with this description for the moment, as this is the most interesting case for us, and has a simpler, direct, algebraic proof.

Before we give a proof of the above theorem, let us meditate on this a bit.
The example \ref{Ex1} is a typical case of part 1 of the above theorem.
But there is a connection with the Sataye Conjecture. Let us consider the following conjecture:\\

\noindent
{\bf Modified Sataye Conjecture:} Let $A:=k[X_1,\ldots,X_{n+1}]$, and let $f\in A$ be such that
$A/(f-\alpha)\cong k[Y_1,\ldots,Y_n]$ for all $\alpha\in\C$.
Then there exist $n$ commuting locally nilpotent derivations $D_1,\ldots,D_n$ on $A$ such that
$A^{D_1,\ldots,D_n}=\C[f]$ and the $D_i$ are linearly independent modulo $(f-\alpha)$ for each  $\alpha\in\C$.\\

\begin{proposition} The Modified Sataye Conjecture is equivalent to the Sataye Conjecture.
\end{proposition}

\begin{proof}
Let us abbreviate the conjectures by SC and MSC.
Suppose we have proven the MSC. Then for any $f$ satisfying ``$A/(f-\alpha)\cong k[Y_1,\ldots,Y_n]$ for all $\alpha\in\C$''
we can find commuting derivations as stated in the MSC.
But using theorem \ref{CDH2} part 2 we get that $f$ is a coordinate in $A$. So the SC is true in that case.

Now suppose we have proven the SC. Let $f$ satisfy the requirements of the MSC , that is,
``$A/(f-\alpha)\cong k[Y_1,\ldots,Y_n]$ for all $\alpha\in\C$''.
Since $f$ satisfies the requirements of the SC,  $f$ then must be a coordinate. So it has $n$ so-called mates:
$\C[f,f_1,\ldots,f_n]=\C[X_1,\ldots,X_{n+1}]$.
But then each of these $n+1$ polynomials $f,f_1,\ldots,f_n$
defines a locally nilpotent derivation, all of them commute, and the intersection of the last $n$ derivations
is $\C[f]$; so the MSC holds.
\end{proof}



But now it is time to stop daydreaming about big conjectures, and start doing some hard-core proofs.
Since the following proof uses the tools of the next section,
the reader is encouraged to read section 3 before reading the following proof in detail.

\begin{proof}{\em (of theorem \ref{CDH2})}
Using lemma \ref{extra1} we have $p_i\in A$ such that $D_j(p_i)=0$ if $i\not = j$, and $D_i(p_i)=q_i(f)\in \C[f]$ of lowest possible
degree. \\
Part 1:
$D_1,\ldots,D_n$ are independent over $A$, but they may become dependent modulo $(f-\alpha)$.
Let us first consider the case where they are independent modulo $(f-\alpha)$: then
$\bar{D}_1,\ldots, \bar{D}_n$ are linearly independent over $A/(f-\alpha)$. Then, by proposition \ref{CDH1}
we have that $A/(f-\alpha)\cong k^{[n]}$.

So, left to prove is that $D_1,\ldots,D_n$ can only  be linearly dependent modulo finitely many $(f-\alpha)$.
But this follows directly from lemma \ref{BLA}, as there are only finitely many zeroes in $q_1q_2\cdots q_n$.

Part 2:
Lemma \ref{BLA} tells us directly that for each $1\leq i\leq n$ and $\alpha\in k$, we have $q_i(\alpha)\not = 0$.
But this means that the $q_i\in k^*$, so the $p_i$ are in fact slices, and using \ref{CDH1a} we are done.
\end{proof}

\section{Tools}

The tools proven in this section focus on the situation of theorem \ref{CDH2} part 1, and are interesting in their own respect.

In this section, $A$ is a $k$-domain, and $\trdeg(A)=n+1 (\geq 1)$.

The following two propositions are proposition 3.2 and 3.4 in \cite{Mau}.

\begin{proposition}\label{CDH1}  Let $D_1,\ldots, D_{n+1}$ be commuting locally nilpotent
$k$-derivations on $A$ which are linearly independent over $A$. Then
\begin{enumerate}
\renewcommand{\theenumi}{(\roman{enumi})}
\item There exist $s_i$ in $A$ such that $D_is_i=\delta_{ij}$ for all $i,j$ and
\item $A=k[s_1,\ldots, s_{n+1}]$ a polynomial ring in $n+1$ variables over $k$.
\end{enumerate}
\end{proposition}

\begin{proposition}\label{CDH1b}
Let $A$ be a {\sc ufd} and let $A^*=k^*$.
Let $D_1,\ldots,D_n$ be commuting
locally nilpotent derivations, linearly independent
over $A$. Then $A^{D_1,\ldots,D_n}=k[f]$ for some $f\in A\backslash k$, and $f-\alpha$ is irreducible for each $\alpha\in\C$.
\end{proposition}

\begin{proposition}\label{CDH1a}
Let $A$, $D_i$, $f$ as in proposition \ref{CDH1b}.
Suppose there exist $s_1,\ldots,s_n$ such that $D_i(s_i)=1$.
Then $A=k[s_1,\ldots,s_n,f]$, a polynomial ring in $n+1$ variables.
\end{proposition}

\begin{proof} This is an easy consequence of the fact that, if $D\in \LND(A)$ having an $s\in A$ such that $D(s)=1$,
then $A^D[s]=A$.
\end{proof}

Define the following abbreviation: \\

\noindent
{\bf (S1:)} Let $A$ be a {\sc ufd} and let $A^*=k^*$.
Let $D_1,\ldots,D_n$ be commuting
locally nilpotent derivations, linearly independent
over $A$.

\begin{lemma}\label{extra1}
Assume (S1).\\
(1)  Then there exist $p_i\in A$ such that $D_j(p_i)=0$ if $j\not = i$, and $D_i(p_i)\in k[f]\backslash \{0\}$.
Furthermore, $k[p_1,\ldots,p_n,f]\subseteq A$ is algebraic.\\
(2) Define $\mathcal{P}_i:=\{p_i\in  A~|~ D_j(p_i)=0$ if $i\not = j$ and $D_i(p_i)\in k[f]$\}.
then $D_i(\mathcal{P}_i)=q_i(f)k[f]$ for some nonzero polynomial $q_i$. Taking $p_i$ such that $D_i(p_i)$ is of lowest possible degree
yields $D_i(p_i)\in k q_i(f)$.
\end{lemma}

\begin{proof}
(1) We assume that all $n$ derivations commute, so $D_i(A^{D_j})\subseteq A^{D_j}$. and therefore $D_i$ sends
$A_i:=A^{D_1,\ldots,D_{i-1},D_{i+1},\ldots,D_n}$ to itself. Taking some $a\in A_i\backslash \C[f]$ nonzero,
we use the fact that $D_i$ is locally
nilpotent to find the lowest $m\in \N$ such that $D^m(a)=0$. Now define $p_i:=D^{m-2}(a)$ (indeed $m\geq 2$).
The rest is easy.\\
(2) Take $p_i$ such that $D_i(p_i)=q_i(f)\not = 0$ has lowest possible degree. Let $\tilde{p}_i\in \mathcal{P}_i$.
then $D_i(\tilde{p}_i)=h_i(f)q_i(f)+r_i(f)$ where $deg(r_i)<deg(q_i)$. Now $D_i(\tilde{p}_i-h_i(f)p_i)=r_i(f)$ so
$r_i=0$. So $D_i(\tilde{p}_i)\in q_i(f)\C[f]$.
\end{proof}

\begin{lemma}\label{BLA}
Assume (S1). Choose $p_i$  such that $D_i(p_i)=q_i(f)$ as in lemma \ref{extra1}, where $q_i$ is of lowest possible degree.
The $D_i$ are linearly dependent modulo $f-\alpha$ if and only if $q_i(\alpha)=0$ for some $i$.
\end{lemma}

\begin{proof}
($\Rightarrow$):
Write ``$\overline{\textup{bars}}$'' for ``modulo $f-\alpha$''.
Suppose that $0\not = D:=g_1D_1+\ldots+g_nD_n$ satisfies $\overline{D}=0$ where $g_i\in A$, and not all $\overline{g}_i=\bar{0}$.
Now $\overline{g}_i\overline{D}_i(\overline{p}_i)=\overline{D}(\overline{p}_i)=\bar{0}$ for each $i$,
so for each $i$, either $\overline{g}_i=\bar{0}$ or $\overline{q_i(f)}=\bar{0}$ (as $f-\alpha$ is irreducible by proposition \ref{CDH1b}).
Since not all $\overline{g}_i=\bar{0}$, at least one $\overline{q_i(f)}=\bar{0}$. Since $f-\alpha$ is irreducible for each $\alpha$,
we not only have $(f-\alpha) | q_i(f)$, but even $(X-\alpha)|q_i(X)$, so $q_i(\alpha)=0$.\\
($\Leftarrow$): Assume $f-\alpha$ divides $q_i(f)$. We need to show that the
$D_i\mod(f-\alpha)$ are linearly dependent over $A/(f-\alpha)$.
Suppose the $\overline{D}_i$ are linearly independent over $\overline{A}$. Then we have $n$ commuting, linearly independent LNDs on
a domain of transcendence degree $n$, so we can use proposition \ref{CDH1}  and conclude that
$\overline{A}^{\overline{D}_1,\ldots,\overline{D}_n}=k$. This means, since $\overline{q_i(f)}=0$, that $\overline{p}_i\in k$.
So, $p_i=(f-\alpha)a+\lambda$ where $a\in A, \lambda\in k$. Now taking $a\in A$ we still have $D_j(a)=0$ for all $j\not = i$,
and $D_i(a)=q_i(f)(f-\alpha)^{-1}\in \C[f]$. This contradicts the assumption that $q_i$ was minimal, so our assumption that
the $\overline{D}_i$ are linearly independent was incorrect.
\end{proof}

Now we want to point out the following phenomenon:

\begin{example}
Let $D_1=Z\partial_X+\partial_Y, D_2=\partial_Y$ on $A=\C[X,Y,Z]$.
Now $A^{D_1,D_2}=\C[Z]$.
The $D_1,D_2$ are linearly independent modulo $Z-\alpha$ as long as $\alpha\not =0$.
But it is clear that a different set of derivations, namely $E_1=\partial_X, E_2=\partial_Y$
commute, their $\C[Z]$-span contains $D_1,D_2$ and the $E_i$ are linearly independent for more fibers $f-\alpha$.
\end{example}

The $E_i$ of the example are an improvement over the $D_i$: all the same properties, but they are linearly independent for more $f-\alpha$.
Perhaps for your given space $A$ and derivations $D_i$ it is impossible to find $E_i$ such that the $E_i$ are independent modulo every $f-\alpha$,
giving more information on your ring $A$. Before we elaborate on this, let us give a lemma that enables construction of the $E_i$:

\begin{lemma} \label{CDH3} Assume (S1).
Define $\mathcal{M}:= k(f)D_1+\ldots+k(f)D_n \cap \DER(A)$.
Then $\mathcal{M}=k[f]E_1\oplus \ldots \oplus k[f]E_n$ for some $E_i\in \mathcal{M}$, and the
$E_i$ have all the properties that the $D_i$ have (i.e. commuting locally nilpotent, linearly independent over $A$).
Furthermore, if the $D_i$ are linearly independent modulo $(f-\alpha)$, then the $E_i$ are too (but not necessary the other way around).
\end{lemma}

\begin{proof}
Use lemma \ref{extra1} we find preslices $p_i$ and $D(p_i)=q_i(f)$ as stated there.

If $D\in \mathcal{M}$ then $D=g_1(f)D_1+\ldots+g_n(f)D_n$
where $g_i(f)\in k(f)$. Now since $D\in \DER(A)$ we have $D(p_i)\in A$. Also $D(p_i)=g_i(f)D_i(p_i)=g_i(f)q_i(f)\in k(f)$
thus $D(p_i)\in A\cap k(f)$, which equals $k[f]$ since $A^*=k^*$.

Therefore the map
$\varphi:\mathcal{M}\lp k[f]^n$  sending $D\lp (D(p_1),\ldots,D(p_n))$ is well-defined.
If $0=\varphi(g_1(f)D_1+\ldots+g_n(f)D_n)$ then $g_i(f)D_i(p_i)=0$ and therefore $g_i(f)=0$; thus $\varphi$ is injective.

Since $\varphi$ is an injective map, $\mathcal{M}$ must be a free $k[f]$-module. Note that
$\mathcal{M}$ can only have dimension $n$. Therefore we can find $E_1,\ldots, E_n$ as required.

Any derivation in $\mathcal{M}$ is locally nilpotent. Even more, any two derivations of $\mathcal{M}$ commute!
Next to that, the $E_i$ are clearly independent over $A$.
\end{proof}

Note that the $E_i$ can be constructively made, given the injective map $\varphi$ in the above proof.
This actually gives an interesting concept. Given the situation (S1), one can improve the derivations $D_i$ (by replacing them by the $E_i$) and then
they are linearly independent modulo as much as possible $f-\alpha$. For every such $\alpha$ we have that $A/(f-\alpha)$ is a polynomial ring.
The question is if the converse holds:\\

\noindent
{\bf Question:}
Assume (S1). Additionally, assume $k[f]D_1+\ldots+k[f]D_n=(k(f)D_1+\ldots+k(f)D_n)\cap \DER(A)$.
Is the set $\{\alpha\in \C~|~ D_1,\ldots,D_n$ linearly dependent modulo $(f-\alpha)\}$ equal to the set
$\{\alpha\in\C~|~A/(f-\alpha)$ is not a polynomial ring$\}$ ? (One always has $\supseteq$.)
Or, if this equality does not hold, what type of rings $A$ do have equality? \\

Note that the requirement ``$A$ UFD'' is absolutely necessary, as for a simple Danielewski surface $\C[X,Y,Z]/(X^2Y-Z^2)$ we find a LND $2Z\partial_Y+X^2\partial_Z$
which is nonzero modulo each $X-\alpha$.(But $A/(f-\alpha)$ is not always a domain in this case, even.)

\section{Unipotent actions}

The authors would like to thank prof. Kraft for pointing out the generalization of theorem \ref{CDH2} part 2, which has become the below theorem \ref{General}.

\begin{proposition}\label{prop}
If $U\times V\lp V$ is an action of a unipotent group $U$ on an affine variety $V$,
then for each $u\in U$, the map $u^*:O(V)\lp O(V)$ is an exponent of a locally nilpotent derivation.
\end{proposition}

For the proof we can refer to proposition 2.1.3 in \cite{EssenBoek}, or ask the reader to verify that $u^*-Id$ is a locally nilpotent endomorphism, and that
thus ``$log(u^*)$'' can be defined, and is a derivation.

This proposition has some immediate consequences, like that the invariants of a unipotent group action are the intersection of kernels of locally nilpotent derivations.
Since kernels of locally nilpotent derivations are factorially closed, their intersection is too, so the invariants of a unipotent group is factorially closed.

In the below theorem, $\C$ is a field of characteristic zero, which is algebraically closed.

\begin{theorem} \label{General} Let $U$ be a unipotent algebraic group of dimension $n$,
acting freely on $X$, a factorial variety of dimension $n+1$ satisfying $\O(X)^*=\C^*$.
Then $X$ is $U$-isomorphic to $U \times \C$. In particular, $X \simeq \C^{n+1}$.
\end{theorem}

\begin{proof}
The fact that $U$ acts free means that each $x\in X$ has trivial stabilizer: $U_x=\{u\in U;ux=x\}=\{\id\}$.
So, each orbit $Ux$ is of dimension $n$. This means that $X\quot U$ is of dimension $1$.
Also, as remarked above, $X^U$ is factorial.
But then it is also normal, and smooth. So $X\quot U$ is a smooth, rational, affine curve,
in other words, an open subvariety of $\C$. Now suppose that $X\quot U \not \cong \C$, so $X\quot U=\C-\{p_1,\ldots,p_n\}$, then $\O(X)^U=\O(\C-\{p_1,\ldots,p_n\})=\C[t,(t-p_1)^{-1},\ldots,(t-p_n)^{-1}]$. This means that $\O(X)$ contains invertible elements
$(t-p_1)^{-1}$, giving a contradiction with the assumption $\O(X)^*=\C^*$.
Hence, $X\quot U\simeq \C$, so $\O(X)^U=\O(X\quot U)=\O(\C)\cong \C[f]$ for some $f$.
Now every $f-\lambda$ ($\lambda \in \C$) is irreducible, as otherwise any irreducible factor of $f-\lambda$ would be in $\O(X)^U$ too.\\
Now consider the map $f:X\lp \C$. This is in fact the map
$X\lp X\quot U$ (as it corresponds to the map $\O(X)\lb \O(X)^U=\C[f]$) and thus surjective.
Also note that the fibers $f^{-1}(\lambda)$ are invariant under $U$: they correspond to the function space $\O(X)/(f-\lambda)$.
By assumption, $U$ acts free on each fiber of $X\lp X\quot U$, which means exactly that $U$ acts free on $f^{-1}(\lambda)$ for each $\lambda$.
Let $x\in f^{-1}(\lambda)$. Then $Ux$ is of dimension $n$ (it is just a copy of $U$). Also, each orbit of a unipotent group is closed (see Satz 4 from \cite{Kraft}),
and therefore
the inclusion $Ux\subseteq f^{-1}(\lambda)$ is an equality.
So orbits of $U$ are the same as fibers of $f$, i.e. we have an orbit fibration (or $U$-fibration).

$X_{\text{sing}}$  is closed and $U$-stable, hence a union of $U$-orbits, and so $\codim X_{\text{sing}} = 1$ or $X_{\text{sing}}$ is empty.
But
$X$ is factorial, so in particular normal, which implies
$codim(X_{sing})\geq 2$. So $X_{\text{sing}}$ is empty, in other words: $X$ is
smooth.\\

Now we claim that $f:X\lp \C$ is smooth. To see this, first note that $\O(f^{-1}(\lambda))=\O(X)/(f-\lambda)$ is reduced as $f-\lambda$ is irreducible, as seen before.
And, as we already implied, the set of functions vanishing on $f^{-1}(\lambda)$ is the ideal $(f-\lambda)$.
Now consider the tangent map $df_x \colon T_x X \to T_\lambda\C=\C$ where $x\in f^{-1}(\lambda)$.
Using ``Satz 2'', page 269 in \cite{Kraft} we see that,
$\ker df \supseteq T_xf^{-1}(\lambda)$, but since $f^{-1}(\lambda)$ is reduced, we even have equality
$\ker df = T_xf^{-1}(\lambda)$.
Now remember that the fiber $f^{-1}(\lambda)$ is an orbit, hence smooth (as any orbit is smooth!). This implies  $\dim T_xf^{-1}(\lambda)=n$ and thus $\dim \ker df=n$.
Since $\dim T_xX = n+1$ we have $\dim Im (df_x)=1$, hence $df_x$ is surjective.
A morphism between smooth varieties is smooth if and only if the differential is surjective. So we have shown that $f$ is smooth.

So: $f:X\lp \C$ is surjective, and smooth.
Let $K:=\ker df\arrowvert_x \subset T_xX$. Take some linear subspace $C$ such that $K\oplus C=T_xX$. Note that $C$ has dimension $1$.
Seeing $X$ as a subset of some $\C^{N}$, we can find hyperplanes $H$ that contains $C$. We even want $H\cap T_xX =C$,
so this means that $H\oplus T_x\subseteq \C^N$, so let us
take a hyperplane $H$ of codimension $n$ such that $H\cap T_xX=C$.
Now let $Z$ be an irreducible component of $H\cap X$ which contains $x$. Also, $dim_x H\cap X\geq 1$, thus $\dim_x Z=1$ and $Z$ is smooth at $x$.
Now $Z$ and $\C$ are smooth, and the differential of $f\|_Z:Z\lp \C$ is an isomorphism at $x$ (implying surjective), thus we  have that $f\|_Z$ is smooth at $x$.
Replacing $Z$, if necessary,  by a (special) open subset $Z'\subset Z$, we have $f\|_Z$ is \'etale.

Now look at the following diagram
$$
\begin{CD}
Z \times_\C X &@>p>> &X\\
@VV{\bar f}V && @VV{f}V \\
Z  &@>{f|_Z}>> &\C
\end{CD}
$$
where $Z \times_\C X =\{(x,z)\in X\times Z~|~ f(x)=f|_Z(z)\}$  is the (schematic) fiber product.
Since $f$ is smooth, the same holds for $\bar f$ and so $Z \times_\C X$ is smooth.
Moreover, $U$ acts on $Z \times_\C X$ by $u(z,x)=(z,ux)$ and
$p(u(x,z))=ux$ ($p$ is $U$-equivariant) and $f(u(x,z))=z=f(x,z)$ ($f$ is $U$-invariant).
The fibers of $\bar{f}$ are $\bar{f}^{-1}(z)=\{(x,z)~|~f(x)=f|_Z(z)\}=\{x~|~ f(x)=\alpha\}=f^{-1}(\alpha)$ where $\alpha=f_Z(z)$.
Now $\bar f$ has a section $\sigma \colon Z\to Z\times _\C X$ given by $z\mapsto (z,z)$, i.e.  $\bar f \circ \sigma = \id_Z$.
Therefore, we can extend the diagram above
$$
\begin{CD}
U \times Z &@>q>>& Z \times_\C X &@>p>> &X\\
@VV{\pr_Z}V && @VV{\bar f}V && @VV{f}V \\
Z &@= &Z  &@>{f|_Z}>> &\C
\end{CD}
$$
where $q \colon U \times Z \to Z \times_\C X$ is given by $(u,z) \mapsto (z, uz)$.
By construction, $q$ is bijective, hence an isomorphism, since the second variety is normal (see \cite{Kraft2} proposition 5.7).
Note that the role of $x$ was arbitrary: for each $x$ we find a neighborhood $Z$ where $Z\times_\C X=Z\times_\C U$.
This last statement exactly means that  the map $f \colon X \to \C$ is a locally trivial principal
$U$-bundle with respect to the \'etale topology:  for every point $\lambda\in\C$ there is an \'etale map $Z \to \C$ such that $\lambda$
is in the image and the fiber product $Z \times_\C X$ is a trivial $U$-bundle, i.e. isomorphic  to $U \times Z \overset{\pr_Z}{\longrightarrow} Z$.

In the paper \cite{Kraft3} we now find a result that tells us that a principal $G$-bundle where $G$ is a unipotent group is trivial over any affine variety, and then we are done.

\end{proof}

\end{document}